\documentclass[11pt,leqno]{amsart}
\usepackage{amssymb,amsmath,amsthm,amsfonts}
\usepackage{graphicx}
\usepackage{xcolor}    

\DeclareMathOperator{\diam}{diam}

\DeclareMathOperator{\id}{id}
\newcommand{\field}[1]{\mathbb{#1}}
\newcommand{\N}{\field{N}}                      
\newcommand{\Z}{\field{Z}}                      
\newcommand{\R}{\field{R}}                      

\newcommand{\de}{\delta}
\newcommand{\la}{\lambda}
\newcommand{\eps}{\epsilon}

\newcommand{\loc}{{\scriptstyle{loc}}}
\newcommand{\ovdimB}{{\overline{\dim_B}}}

\def\Barint_#1{\mathchoice
          {\mathop{\vrule width 6pt height 3 pt depth -2.5pt
                  \kern -8pt \intop}\nolimits_{#1}}%
          {\mathop{\vrule width 5pt height 3 pt depth -2.6pt
                  \kern -6pt \intop}\nolimits_{#1}}%
          {\mathop{\vrule width 5pt height 3 pt depth -2.6pt
                  \kern -6pt \intop}\nolimits_{#1}}%
          {\mathop{\vrule width 5pt height 3 pt depth -2.6pt
                  \kern -6pt \intop}\nolimits_{#1}}}

\theoremstyle{plain}
\newtheorem{theorem}{Theorem}
\newtheorem{theoremA}{Theorem}
\newtheorem{corollary}[theorem]{Corollary}

\newtheorem{lemmaA}{Lemma}
\newtheorem{proposition}[theorem]{Proposition}

\theoremstyle{definition}
\newtheorem{definition}[theorem]{Definition}

\newtheorem{remark}[theorem]{Remark}

\newtheorem{question}[theorem]{Question}

\numberwithin{theorem}{section} \numberwithin{equation}{section}

\title{Sobolev mappings on metric spaces and Minkowski dimension}
\author[Efstathios-K. Chrontsios-Garitsis]{Efstathios-K. Chrontsios-Garitsis}
\address{Department of Mathematics \\ University of Tennessee, Knoxville \\ 1403 Circle Dr \\ Knoxville, TN 37966}
\email{echronts@utk.edu, echronts@gmail.com}

\begin{document}

\begin{abstract}
We introduce the class of compactly H\"older mappings between metric spaces and determine the extent to which they distort the Minkowski dimension of a given set. These mappings are defined purely with metric notions and can be seen as a generalization of Sobolev mappings, without the requirement for a measure on the source space. In fact, we show that if $f:X\rightarrow Y$ is a continuous mapping lying in some super-critical Newtonian-Sobolev space $N^{1,p}(X,\mu)$, under standard assumptions on the metric measure space $(X,d,\mu)$, it is then a compactly H\"older mapping. The dimension distortion result we obtain is new even for Sobolev mappings between weighted Euclidean spaces and generalizes previous results of Kaufman \cite{Kaufman} and Bishop-Hakobyan-Williams \cite{BishHakWill:QSdistortion}.
\end{abstract}

\maketitle

\section{Introduction}
There has been growing interest in fractals within pure and applied mathematics, especially for the past three decades.
One of the core subjects of fractal geometry is to determine and study notions of dimension which provide more insight on the structure of fractal sets (see \cite{FalcBook} for a thorough exposition). The Minkowski, or box-counting dimension, is such a notion, which has been popular within many fields of research, such as partial differential equations (PDEs) \cite{MinkNavierStokesEq}, signal processing \cite{MinkSignal} and mathematical physics \cite{MinkPhysics}. In fact, many manuscripts in applied fields simply refer to the Minkowski dimension as ``fractal dimension", since it is very often the only dimension notion that suits that context.

An essential tool in the areas of partial differential equations and calculus of variations is the notion of Sobolev mappings, due to certain equations admitting solutions only in a weak sense. The increasing interest and the need to extend such analytical notions and tools to metric spaces soon emerged. This has been a very active line of research for the past two and a half decades with various applications, such as developing the theory of  PDEs \cite{KigamiAnFr}, calculus of variations \cite{AmbrosioBV} and optimal transportation \cite{AmbrosioGradFlow} to the non-smooth setting of fractal spaces. The theory of Sobolev-type mappings defined between metric spaces has been developed by many authors (for instance \cite{CheegerSob}, \cite{HajSob}, \cite{HajKoskSob}, \cite{KorSchSob}, \cite{NagesNewtonianSob}), who have used different ideas to adjust the theory to different settings. We refer to the book by Heinonen-Koskela-Shanmugalingam-Tyson \cite{Juha-Jeremy-etc-book} for a detailed exposition. 

A question of broad interest has been to determine in what ways certain classes of mappings distort  dimension notions. One of the earliest results in this direction is by Gehring-V\"ais\"al\"a \cite{GehringVais}, who gave quantitative bounds on how quasiconformal mappings, a special class of super-critical Sobolev mappings (see Section \ref{sec:background}), change the Hausdorff dimension of a subset of $\R^n$. Kaufman later proved similar bounds for the distortion of the Hausdorff and Minkowski dimensions under general super-critical Sobolev mappings \cite{Kaufman}. Since the aforementioned two manuscripts, the study of dimension distortion has been extended to sub-critical Sobolev mappings \cite{SobSubcriticalDimDist1}, \cite{SobSubcriticalDimDist2}, to other dimension notions \cite{OurQCspec}, \cite{SobSubcritDimDist0}, \cite{HolomSpecChron}, and to other settings, such as distortion by Sobolev and quasisymmetric mappings defined on manifolds and metric spaces  \cite{btw:heisenberg}, \cite{bmt:grassmannian}, \cite{BaloghAGMS}, \cite{BishHakWill:QSdistortion}. However, in the non-Euclidean setting all results are regarding the Hausdorff dimension, or cases where all dimensions coincide, and the distortion of the Minkowski dimension has so far not been determined outside the Euclidean case.

Motivated by this rich theory and the above open direction, we introduce a class of mappings between metric spaces which resembles that of Sobolev-type mappings, and study how they distort the Minkowski dimension. We call these mappings ``compactly H\"older", due to the improved H\"older condition they satisfy on coverings of compact sets. This class is contained in the locally H\"older class and, under standard assumptions, it contains continuous Newtonian-Sobolev and quasisymmetric mappings defined between metric  spaces. Given two metric spaces $X$, $Y$, and constants $p>1$, $\alpha\in(0,1)$, a mapping $f:X\rightarrow Y$ is \textbf{$(p,\alpha)$-compactly H\"older} if for any compact set $E\subset X$ and any 
covering of $E$ by balls, the $\alpha$-H\"older coefficients of $f$ on the balls are $p$-summable in a uniform way, as long as the balls are small enough and do not overlap too much (see Definition \ref{def: CH maps}). While this is a broader class of mappings, one promising trait compared to other Sobolev-type notions is the fact that it can be defined using purely metric tools. This provides  independence from measures that $X$ and $Y$ might be equipped with. 

Our first main result is on the distortion of the Minkowski dimension under compactly H\"older mappings.
\begin{theorem}\label{thm:main_Holder}
    Suppose $(X,d)$ is a doubling metric space and $(Y,d_Y)$ is an arbitrary metric space. For $p>1$ and $\alpha\in(0,1)$, if $f:X\rightarrow Y$ is $(p,\alpha)$-compactly H\"older and $E\subset X$ is bounded with $\dim_B E =d_E$, then
		\begin{equation}\label{eq:CHBoxdistortion}
			\dim_B f(E) \leq \frac{pd_E}{\alpha p+d_E}.
		\end{equation}
\end{theorem}

Note that \eqref{eq:CHBoxdistortion} is an improvement upon the trivial bound $\dim_B f(E)\leq d_E/\alpha$ for $\alpha$-H\"older mappings (see \cite{FalcBook} Ex.~2.2). In fact, due to \eqref{eq:CHBoxdistortion} we know that for $\alpha\in (0,1)$, the $\alpha$-H\"older snowflake mapping $id_\alpha: ([-1,1],d_{\text{euc}})\rightarrow ([-1,1],d_{\text{euc}}^\alpha)$ is not compactly H\"older for any $p$, since $\dim_B id_\alpha([0,1])=\alpha^{-1}$.

We show that continuous mappings lying in the class of Newtonian-Sobolev mappings are compactly H\"older for appropriate constants $p$ and $\alpha$. This facilitates the study of the Minkowski dimension distortion due to \eqref{eq:CHBoxdistortion}. The assumptions on the metric measure space $(X,d,\mu)$ below are standard in the context of analysis on metric spaces (see Section \ref{sec:background}).

\begin{theorem}\label{thm:NewtonDim}
	Suppose $(X,d,\mu)$ is a proper, locally $Q$-homogeneous
    metric measure space supporting a $Q$-Poincar\'e inequality, and $(Y,d_Y)$ is an arbitrary metric space.
    Let $f:X\rightarrow Y$ be a continuous mapping with an upper gradient $g\in L^p_\loc(X)$ for $p>Q$. Then $f$ is $(q,1-Q/q)$-compactly H\"older for all $q\in (Q,p)$. Moreover, if $E\subset X$ is bounded with $\dim_B E =d_E<Q$, then
		\begin{equation}\label{eq:SobBoxdistortion}
			\dim_B f(E) \leq \frac{p d_E}{p-Q+d_E}<Q.
		\end{equation}
		  
\end{theorem}

The Newtonian-Sobolev class constitutes one of the broader classes of Sobolev-type mappings between metric spaces (see Theorem 10.5.1 in \cite{Juha-Jeremy-etc-book}). This generalizes the result of Kaufman \cite{Kaufman} and settles the Minkowski dimension distortion problem on metric spaces, by providing a quantitative bound similar to that in the usual Eucledian setting. 
A non-exhaustive list of spaces where the above result could be applied includes Carnot groups, Laakso spaces, Gromov hyperbolic groups and boundaries (see Chapter 14 in \cite{Juha-Jeremy-etc-book} and the references therein). 

It should be noted that the bound \eqref{eq:SobBoxdistortion} is new even for weighted Euclidean spaces.  For instance, if $\la_n$ is the $n$-Lebesgue measure, the conditions of Theorem \ref{thm:NewtonDim} are satisfied by the weighted Euclidean metric measure space $(\R^n,d_{\text{euc}}, w\la_n)$, for a wide variety of weights $w:\R^n\rightarrow [0,\infty]$, such as the class of Muckenhoupt weights (see Chapter 1 in \cite{HeinonenWeightedSobBook}). These weights were introduced by Muckenhoupt \cite{MuckOrigin} in order to characterize the boundedness of the Hardy-Littlewood maximal operator on weighted $L^p$ spaces, and have since established an active area within Functional and Harmonic Analysis (see \cite{GrafClassMucken}, \cite{GrafModMucken}). In fact, certain Muckenhoupt weights have recently been associated with the Minkowski dimension through the notion of ``weak porosity" (see \cite{CarlosEtAlWeakPoros}, \cite{CarlosWeakPorosMetric} for details).

Theorem \ref{thm:main_Holder} also provides control on the distortion of the Minkowski dimension under quasisymmetric mappings. 

\begin{corollary}\label{cor:QS_Dim}
	Suppose $Q>1$ and $(X,d,\mu)$, is a proper, $Q$-Ahlfors regular metric measure space that supports a $p_0$-PI for $p_0\in(1,Q)$, and $(Y,d_Y)$ is a $Q$-Ahlfors regular metric space. Let $f:X\rightarrow Y$ be an $\eta$-quasisymmetric homeomorphism. If $E\subset X$ is bounded with $\dim_B E =d_E\in (0,Q)$, then
		\begin{equation}\label{eq:QSboxdistortion}
			0< \frac{(p-Q)d_E}{p-d_E}  \leq \dim_B f(E) \leq \frac{p d_E}{p-Q+d_E}<Q,
		\end{equation} where $p>Q$ only depends on $\eta(1)$, $\eta^{-1}(1)$.
		  
\end{corollary}


Bishop-Hakobyan-Williams \cite{BishHakWill:QSdistortion} studied this problem  in the case where the input set $E$ is Ahlfors regular, which implies that all dimension notions for $E$ coincide. Their motivation was the absolute continuity on lines  property (ACL) that quasisymmetric mappings satisfy in the Euclidean setting. Their result provides a fundamental generalization of this fact in the metric measure spaces setting, in the sense that an Ahlfors regular set can be considered a generalization of a line. In general, however, we could have the Hausdorff and Minkowski dimensions of $E$ to differ. In such a case, the results from \cite{BishHakWill:QSdistortion} cannot be applied, while \eqref{eq:QSboxdistortion} provides quantitative bounds on $\dim_B f(E)$ and the result of Balogh-Tyson-Wildrick (Theorem 1.1 in \cite{BaloghAGMS}) provides similar bounds on $\dim_H f(E)$.

This paper is organized as follows. Section \ref{sec:background} reviews the required background on metric measure spaces, along with the notions of the Minkowski dimension, Newtonian-Sobolev and quasisymmetric mappings. A proof of a crucial characterization of the Minkowski dimension using dyadic cubes in metric spaces is also provided.  In Section \ref{sec:DimDistProofs} we prove our main results, namely Theorems \ref{thm:main_Holder}, \ref{thm:NewtonDim} and Corollary \ref{cor:QS_Dim}. Section \ref{sec:Final Remarks} contains further remarks and future directions motivated by this work.

~

\paragraph{\bf Acknowledgments.} The author wishes to thank Carlos Mudarra and Vyron Vellis for the fruitful conversations, and especially Jeremy Tyson for introducing him to the  area of analysis on metric spaces.

\section{Background}\label{sec:background}

\subsection{Metric spaces and dimensions.}\label{subsec:MS}
Let $(X,d)$ be a metric space. We  use the Polish notation $d(x,y)=|x-y|$ for all $x,y \in X$ and denote the open ball centered at $x$ of radius $r>0$ by
$$
B(x,r):= \{z\in X: \,\, |x-z|<r \}.
$$

Given a ball $B=B(x,r) \subset X$, we  denote by $\lambda B$ the ball $B(x,\lambda r)$, for $\lambda>0$. We say that $(X,d)$ is a \textit{doubling metric space} if there is a \textit{doubling constant} $C_d\geq 1$ such that for every $x\in X$, $r>0$, the smallest number of balls of radius $r$ needed to cover $B(x,2r)$ is at most $C_d$. Note that the doubling property implies that $X$ is a separable metric space.

Let $E$ be a bounded subset of $X$. For $r>0$, denote by $N(E,r)$ the smallest number of sets of diameter at most $r$ needed to cover $E$. The {\it (upper) Minkowski dimension} of $E$ is defined as
$$
\ovdimB(E) = \limsup_{r\to 0} \frac{\log N(E,r)}{\log(1/r)}.
$$
This notion is also known as \textit{upper box-counting dimension}, which justifies the notation with the subscript ``B" typically used in the literature (see \cite{FalcBook}, \cite{FraserBook}). We drop the adjective `upper' and the bar notation throughout this paper as we will make no reference to the lower Minkowski dimension. For any fixed $r_0\leq \diam E$, an equivalent formulation is
$$
\dim_B(E) = \inf \{d>0 \,:\, \exists\,C>0\mbox{ s.t. } N(E,r) \le C r^{-d} \mbox{ for all $0<r\le r_0$} \}.
$$

On Euclidean spaces $X=\R^n$ with the usual metric one can use dyadic cubes instead of arbitrary sets of diameter at most $r$ to define the Minkowski dimension (see \cite{FraserBook}, \cite{FalcBook}). On arbitrary metric spaces, however, there are various generalizations of dyadic cube constructions. One of the first manuscripts addressing this idea was by David  \cite{CubesGuyC1} (see also \cite{CubesGuyC2}), while one of the first explicit constructions of a system of dyadic cubes is due to Christ \cite{ChristCubes} (see also \cite{ChristCubesBook}). Other important dyadic cube constructions on metric spaces can be found in \cite{MoreCubes1}, \cite{Hyt:dyadic}, \cite{MoreCubes2}, \cite{MoreCubes3}, which is by no means an exhaustive list. We believe the most fitting  notion for our context to be the one due to Hyt\"onen and Kairema.
\begin{theoremA}[Hyt\"onen, Kairema \cite{Hyt:dyadic}]\label{thm:Dydadic}
	Suppose $(X,d)$ is a doubling metric space. Let $0<c_0\leq C_0<\infty$ and $\delta\in (0,1)$ with $12 C_0 \delta \leq c_0$. For any non-negative $k\in \Z$ and collection of points $\{ z_i^k \}_{i\in I_k}$ with
	
	\begin{equation}\label{eq:centers_away}
		|z_i^k-z_j^k|\geq c_0 \de^k, \,\,\,\, \text{for} \,\, i\neq j
	\end{equation}
	and
	\begin{equation}\label{eq:points_close_centers}
	\min_i |z_i^k-x|< C_0 \de^k, \,\,\,\, \text{for all} \,\, x\in X
	\end{equation}
	we can construct a collection of sets $\{ Q_i^k \}_{i\in I_k}$ such that
	\begin{itemize}
		\item[(i)] if $l \geq k$ then for any $i\in I_k$, $j\in I_l$ either $Q_j^l\subset Q_i^k$ or $Q_j^l \cap Q_i^k=\emptyset$,
		\vspace{0.1cm}
		\item[(ii)] $X$ is equal to the disjoint union $\bigcup\limits_{i\in I_k} Q_i^k$, for every $k\in\N$
		\vspace{0.1cm}
		\item [(iii)] $B(z_i^k, c_0 \de^k /3) \subset Q_i^k \subset B(z_i^k, 2C_0 \de^k)=:B(Q_i^k)$ for every $k\in \N$,
		\vspace{0.1cm}
		\item [(iv)] if $l\geq k$ and $Q_j^l\subset Q_i^k$, then $B(Q_j^l)\subset B(Q_i^k)$.
	\end{itemize}
	
	For non-negative $k\in \Z$, we call the sets $Q_i^k$ from the construction of Theorem \ref{thm:Dydadic} ($\de$-)\textit{dyadic cubes} of level $k$ of $X$.
	
\end{theoremA}

Fix $\delta$, $c_0$ and $C_0$ as in Theorem \ref{thm:Dydadic}. Moreover, for every non-negative $k\in \Z$ we fix a collection of points $\{ z_i^k \}_{i\in I_k}$ and the corresponding collection of $\delta$-dyadic cubes $Q^k_i$. To see why such a collection of points exists, consider the covering $\{ B(z,c_0 \delta^k): z\in X \}$ of $X$ and apply the $5B$-covering lemma. By separability of $X$ and by choosing $c_0$ and $C_0$ so that $5c_0 \delta^k<C_0 \delta^k$, the existence of centers $\{ z_i^k \}_{i\in I_k}$ is ensured. We fix such a system of dyadic cubes for the rest of the paper and denote by $N_k(E)$ the number of dyadic cubes of level $k$ that intersect the set $E\subset X$.

\begin{proposition}\label{prop:boxdyadic}
Let $E \subset X$ be a bounded subset and $k_E\in \N$ a fixed integer for which $\delta^{k_E}\leq \diam E$. Then
$$
\dim_B(E) = \inf \{d>0 \,:\, \exists\,C>0\mbox{ s.t. } N_k(E) \le C \de^{-k d} \mbox{ for all integers $k\geq k_E$} \}.
$$
\end{proposition}

\begin{proof}
	Fix a positive integer $k_E \geq \log(\diam E)/\log \de$. Let $$A=\{d>0 \,:\, \exists\,C>0\mbox{ s.t. } N(E,r) \le C r^{-d} \mbox{ for all $0<r\le\diam E$} \}$$ and $$\Delta =  \{d>0 \,:\, \exists\,C>0\mbox{ s.t. } N_k(E) \le C \de^{-k d} \mbox{ for all integers $k\geq k_E$} \}.$$ 
 
    We will first show that for all $k\geq k_E$ we have
	\begin{equation}\label{eq:Nrelations}
		N(E,4C_0 \delta^k) \leq N_k(E) \leq C' N(E,\de^k),
	\end{equation} where $C'=C_d \left( \frac{c_0}{3(4C_0+1)} \right)^{-\log_2 C_d}$ depends only on the dyadic cube constants and the doubling constant $C_d$ of the space $X$.

	The left-hand side of \eqref{eq:Nrelations} is trivial, since $Q_i^k \subset B(z_i^k, 2C_0 \de^k)$, which implies that $\diam Q_i^k \leq 4C_0 \de^k$ as needed.
	
	Let $U$ be a set of diameter at most $\de^k$ and set $I_k^U:= \{ i\in I_k : Q_i^k \cap U\neq \emptyset \}$. Denote by $\mathcal{Q}_U$ the union of all $Q_i^k$ for all $i\in I_k^U$ and fix $x_0 \in \mathcal{Q}_U$.
	
	Suppose $x_0 \in U$ and let $x \in  \mathcal{Q}_U$. If $x\in U$, then $|x_0-x|\leq \de^k$, due to $\diam U\leq \de^k$. If $x\notin U$, it lies in some $Q_{i_x}^k$ for some $i_x\in I_k^U$ and there is some $x'\in U\cap Q_{i_x}^k$ such that
	$$
	|x_0-x|\leq |x_0-x'|+|x'-x| \leq \diam U + \diam Q_{i_x}^k \leq \de^k + 4C_0 \de^k.
	$$ Note that the existence of such $x'$ is guaranteed by the definition of $\mathcal{Q}_U$, which is the union of all cubes of level $k$ intersecting $U$. Hence, if $x_0\in U$, we get that $|x_0-x| \leq \de^k + 4C_0 \de^k$ for all $x \in\mathcal{Q}_U$.
	
	Similarly, repeating the above arguments in the case where $x_0 \in \mathcal{Q}_U \setminus U$ also implies $|x_0-x| \leq \de^k + 4C_0 \de^k$ for all $x \in  \mathcal{Q}_U$. Since $x_0\in  \mathcal{Q}_U$ was arbitrary in both cases, we have shown
	$$ 
	\diam (\mathcal{Q}_U)\leq \de^k (1+4C_0).
	$$
	
	As a result, the set $\mathcal{Q}_U$ lies in $B(y, (4C_0+1)\de^k)$ for some $y\in \mathcal{Q}_U$. But each cube $Q_i^k$ includes a ball $B(z_i^k, c_0 \de^k /3)$ and all these are disjoint by the choice of the points $\{ z_i^k \}$ in Theorem \ref{thm:Dydadic}. By the doubling property of $X$, these disjoint balls of radius $ c_0 \de^k /3$ that lie in a larger ball of radius $(4C_0+1)\de^k$ can be at most $C'=C_d \left( \frac{c_0}{3(4C_0+1)} \right)^{-\log_2 C_d}$ in number.
	
	Hence, if $E$ is covered by $N(E,\de^k)$ sets of diameter at most $\de^k$, we can cover each such $U$ by at most $C'$ many $\de$-dyadic cubes of level $k$, resulting in a collection that consists of  at most $C' N(E,\de^k)$ cubes in total and covers $E$. This implies $N_k(E) \leq C' N(E,\de^k)$ as needed.
	
	We will now use \eqref{eq:Nrelations} to prove that $A=\Delta$. Let $d \in A$ and $k\geq k_E$. Then there exists a constant $C$ such that $N(E,r)\leq C r^{-d}$ for $r=\de^k$. By the right-hand side of \eqref{eq:Nrelations}, $N_k(E)\leq C' C r^{-d}= C' C\de^{-k d}$, which implies that $d \in \Delta$. Since $d$ was arbitrary, the inclusion $A\subset \Delta$ is proved.
	
	Now let $d \in \Delta$ and $r \in (0, \diam E)$. Let $k\geq k_E$ be the smallest integer such that $4\de^k C_0 \leq r < 4\de^{k+1} C_0$. This, along with \eqref{eq:Nrelations}, imply
	$$
	N(E,r)\leq N(E,4C_0\de^k) \leq N_k(E).
	$$ But $N_k(E) \leq C \de^{-k d} \leq C(4 \de C_0)^d r^{-d}$ by choice of $k$ and $d$. Hence, we have shown that
	$$
	N(E,r)\leq C (4\de C_0)^d r^{-d},
	$$ which implies that $d \in A$. Since $d$ was arbitrary, we have shown that $\Delta$ is also a subset of $A$, concluding the proof.
\end{proof}

\begin{remark}
The construction of dyadic cubes in \cite{Hyt:dyadic} was actually given for quasimetric spaces. As a result, Proposition \ref{prop:boxdyadic} is also true if $X$ is a quasimetric doubling space. The proof is almost identical, with the only difference being the dependence of a few of the constants on the quasimetric constant of the space.
\end{remark}


\subsection{Mappings between metric spaces}\label{subsec:BackSobolev}
Given $\alpha\in (0,1)$, a mapping $f:X\rightarrow Y$ and a set $B\subset X$, we  define the \textit{$\alpha$-H\"older coefficient} of $f$ on $B$ as
$$
|f|_{\alpha, B}:= \sup\left\{ \frac{|f(x)-f(y)|}{|x-y|^\alpha}: \, x, y \in B \,\, \text{distinct} \right\}.
$$ If $|f|_{\alpha, B}<\infty$ then we say that $f$ is \textit{$\alpha$-H\"older continuous} in $B$. 

Given an at most countable index set $I$, we denote by $\ell^p(I)$ the space of real-valued sequences $\{c_i\}_{i\in I}$ with finite $p$-norm $(\sum_{i\in I} c_i^p)^{1/p}<\infty$. We call $\sum_{i\in I} c_i^p$ the \textit{$p$-sum} of the sequence $\{c_i\}_{i\in \N}$.

For the rest of the paper, all index sets are assumed to be at most countable. We are now prepared to introduce the class of compactly H\"older mappings.
\begin{definition}\label{def: CH maps}
    Let $f:X\rightarrow Y$ be a mapping between two arbitrary metric spaces. For $p>1$, $\alpha\in(0,1)$, we say $f$ is $(p,\alpha)$\textbf{-compactly H\"older}, and write $f\in CH^{p,\alpha}(X,Y)$, if for any compact set $E\subset X$ and any $\eps\in (0,1)$ there are $r_E>0$ and $C_E>0$ satisfying the following: \\
if $\{B_i\}_{i\in I}$ is a collection of balls $B_i:=B(x_i,r)$ with $x_i\in X$, $r<r_E$ that covers $E$ and $B(x_i,\eps r)\cap B(x_j,\eps r)=\emptyset$ for all distinct $i, j\in I$, then the $p$-sum of the H\"older coefficients of $f$ on $B_i$ is at most $C_E$, i.e.,
\begin{equation}\label{eq: CH-def-inequality}
    \sum\limits_{i\in I} |f|_{\alpha, B_i}^p\leq C_E.
\end{equation}
\end{definition}
Here we follow the convention that if $\{B(x_i,r)\}_{i\in I}$ covers $E$, it is implied that $B(x_i,r)\cap E\neq\emptyset$ for all $i$, but not all $x_i$  necessarily lie in $E$. Note that applying the definition on singleton sets yields that compactly H\"older mappings are locally H\"older continuous, and, hence, uniformly continuous on compact sets. Moreover, in the setting of Definition \ref{def: CH maps} it is actually implied by \eqref{eq: CH-def-inequality} that there are $C_i>0$ such that
\begin{equation}\label{eq: CH-inequality}
    \diam f(B(x_i,r))\leq C_i (\diam B(x_i,r))^\alpha
\end{equation} with $\sum_{i\in I}C_i^p\leq C_E$. This inequality is crucial in the proof of our dimension distortion results and provides more insight on the relation with the Euclidean setting.

The motivation for Definition \ref{def: CH maps} comes from continuous super-critical Sobolev maps between Euclidean spaces, i.e. continuous maps in $W^{1,p}(\Omega;\R^n)$ with $\Omega\subset \R^n$ and $p>n$. These maps satisfy the Morrey-Sobolev inequality on balls and cubes [see, for instance, \cite{EvansPDE} p. 280 Thm 4 and p.~283 Remark]. This inequality implies the local H\"older continuity of the map with exponent $1-n/p$, although it is in fact a stronger property. One way to observe that is to apply the inequality on all dyadic cubes that cover a compact subset $E$ lying in $\Omega$. Then all the resulting inequalities resemble \eqref{eq: CH-inequality} with $C_E$ being the $L^p$-norm of the gradient on $E$ times a uniform constant, and the sequence $C_i$ being the $L^p$ norm of the gradient on each cube times a uniform constant. 

For the Sobolev-type mappings we discuss next we need a measure for our space. A triplet $(X,d,\mu)$ is called a metric measure space if $(X,d)$ is separable and $\mu$ is a Borel measure on $X$ that assigns a positive and finite value on all balls in $X$. Thus, throughout the paper all measures are considered to have the aforementioned properties, even if not stated explicitly. For $p\in (0,\infty]$ we denote the space of \textit{$p$-integrable} real-valued functions defined on $X$ by $L^p(X,\mu)$, or simply by $L^p(X)$ if the measure follows from the context. We also denote by $L^p_\loc(X)$ the space of locally $p$-integrable real-valued functions defined on $X$. Moreover, for a ball $B\subset X$ and $u\in L^1(B)$ we denote by $u_B$ the average of $u$ over $B$, i.e., $u_B:= \Barint_{B}u d\mu = \mu(B)^{-1} \int_{B} u d\mu$.

The notion of upper gradients is necessary for the definition of certain types of Sobolev spaces. We say that a Borel function $g:X\rightarrow [0,\infty]$ is an \textit{upper gradient} of a continuous map $f:X\rightarrow Y$ if for every rectifiable curve $\gamma:[0,1]\rightarrow X$ we have
$$
|f(\gamma(0))-f(\gamma(1))|\leq \int_\gamma g ds.
$$ The notion was introduced by Heinonen and Koskela in \cite{HeinKoskQCbegin} under a different name, and was employed by Shanmugalingam in \cite{NagesPHDNewtonianSob} and \cite{NagesNewtonianSob} in order to define an appropriate notion of Sobolev mappings in the metric measure space context. More specifically, for $p>1$, by viewing $Y$ as a subset of a Banach space under the Kuratowski-Fr\'echet isometric embedding (see for instance p.~105-106 in \cite{Juha-Jeremy-etc-book}), the \textit{Newtonian-Sobolev space} $N^{1,p}(X;Y)$ is defined as the collection of equivalence classes of mappings $f:X\rightarrow Y$ in $L^p(X;Y)$ with an upper gradient in $L^p(X)$. See \cite{Juha-Jeremy-etc-book} for a thorough exposition and relations of this Sobolev space notion with others defined in similar settings.

One weakness of the aforementioned definition of upper gradients is the dependence on rectifiable curves of $X$. To ensure there are enough such curves, the following property is typically assumed for the source metric measure space $(X,d,\mu)$ in this context (see \cite{hei:lectures}, \cite{Juha-Jeremy-etc-book}).

We say that a metric measure space $(X,d, \mu)$ \textit{supports a $p$-Poincar\'e inequality with data $\la$ } if there are $p>0$, $C>0$ and $\lambda\geq 1$ such that  if $u:X\rightarrow \R$ is a function with upper gradient $g:X\rightarrow [0,\infty]$, then
	\begin{equation}
		\Barint_{B} |u-u_B| d\mu \leq C \diam B \left( \Barint_{\lambda B} g^p d\mu  \right)^{1/p},
	\end{equation}
for every open ball $B$ in $X$. If the data $\la$ is implied by the context, we say $X$ \textit{supports a $p$-PI}.

The following notions for measures are also typically needed in this setting. We say a metric measure space $(X,d,\mu)$ is \textit{locally $Q$-homogeneous}, for some $Q>0$, if for every compact set $K\subset X$ there are constants $\tilde{R}_{\text{hom}}(K)>0$, $\tilde{C}_{\text{hom}}(K)\geq 1$ such that
$$
\frac{\mu(B(x,r_2))}{\mu(B(x,r_1))}\leq \tilde{C}_{\text{hom}}(K) \left(\frac{r_2}{r_1}\right)^Q,
$$ for all $x\in K$ and scales $0<r_1<r_2<\tilde{R}_{\text{hom}}(K)$. We say a metric space $(X,d)$ is \textit{locally $Q$-homogeneous} if there is a measure $\mu$ on $X$ such that $(X,d,\mu)$ is locally $Q$-homogeneous. One particular property due to local homogeneity that we  need is the lower bound on the measure of a ball by its radius to a power. More specifically, for $R_{\text{hom}}(K)=\tilde{R}_{\text{hom}}(K)/3$ and a potentially larger constant $C_{\text{hom}}(K)$, which still only depends on $K$, it can be shown that 
\begin{equation}\label{eq:lower-AR}
    \frac{r^Q}{C_{\text{hom}}(K)} \leq \mu (B(x,r)),
\end{equation} for all $x\in K$ and $r\in (0,R_{\text{hom}}(K))$.

On the other hand, it can be necessary at times to also have a similar upper bound on the measure. We say that $(X,d,\mu)$ is \textit{$Q$-Ahlfors regular} for some $Q>0$ if there is a constant $C_A>0$ such that for all $x\in X$ and all $r\in(0,\diam X)$ we have
    $$
    \frac{1}{C_A}r^Q \leq \mu (B(x,r)) \leq C_A r^Q.
    $$ 
    We say a metric space $(X,d)$ is \textit{$Q$-Ahlfors regular} if there is a measure $\mu$ on $X$ such that $(X,d,\mu)$  is $Q$-Ahlfors regular. Note that $Q$-regularity of a measure implies the $Q$-homogeneous property.

The above notions are especially useful in establishing a connection between Newtonian-Sobolev and quasisymmetric mappings. Given a homeomorphism $\eta:[0,\infty)\rightarrow [0,\infty)$, an embedding $f:X\rightarrow Y$ is \textit{$\eta$-quasisymmetric} if for all distinct $x, y, z\in X$ we have
$$
\frac{|f(x)-f(y)|}{|f(z)-f(y)|}\leq \eta\left( \frac{|x-y|}{|z-y|} \right).
$$ In fact, under the assumption that $X$ and $Y$ are both locally $Q$-homogeneous, an \textit{analytic} version of the above definition can be derived involving the space $N^{1,Q}(X;Y)$, similar to the analytic definition in the Euclidean setting $\R^n$ involving the usual Sobolev space $W^{1,n}(\R^n;\R^n)$ (see Theorem 9.8 in \cite{QCanalyticDefHKSTyson}). We do not state that version, since more notions would be needed and we will not use it directly.

\section{Distortion of dimensions}\label{sec:DimDistProofs}

\subsection{Compactly H\"older mappings}\label{subsec:CH maps}
Suppose $Y$ is an arbitrary metric space and $X$ is a doubling and proper metric space with a fixed system of dyadic cubes as in Section \ref{sec:background}. Since all mappings considered are  defined in $X$ and into $Y$, we set $CH^{p,\alpha}=CH^{p,\alpha}(X,Y)$. Let $f\in CH^{p,\alpha}$ for $p>1$, $\alpha\in (0,1)$ and $E\subset X$ be a bounded set. We plan on using the system of dyadic cubes to cover $E$, Proposition \ref{prop:boxdyadic} and $d_E=\dim_B E$ to count how many we need at each level $k$, and map them with $f$ into $Y$ to cover $f(E)$. Inequality \eqref{eq: CH-inequality} is crucial to this argument in order to control the size of the images of cubes under $f$, indicating how many times we have to sub-divide the cubes further so that we achieve small enough images for the covering of $f(E)$.

\begin{proof}[Proof of Theorem \ref{thm:main_Holder}]
		Suppose $d_E<\frac{pd_E}{\alpha p+d_E}$ and let $d\in (d_E, p-\alpha p)$, and $D := \frac{pd}{\alpha p+d}$. The proof is analogous in the case $d_E\geq \frac{pd_E}{\alpha p+d_E}$ by letting $d>d_E$.
  
    To prove \eqref{eq:CHBoxdistortion} in this case it is enough to show that there is some $C'>0$ such that for all $r \in (0, \diam f(E))$ we have $N(f(E),r)\leq C' r^{-D}$. This would imply  $\dim_B f(E) \leq D$ for all $d\in (d_E, p-\alpha p)$ and the desired upper bound is achieved by taking $d\rightarrow d_E$.

  
		By stability of the Minkowski dimension under closure (see p.~18, 20 in \cite{FraserBook}), we can assume without loss of generality that $E$ is closed. Since $X$ is proper, this implies that $E$ is compact.
		Denote by $k_E\in \N$ the smallest integer such that $\de^{k_E} \leq r_E/4 C_0$, where $r_E$ is as in  Definition \ref{def: CH maps}. This choice is to ensure that we focus on levels $k\geq k_E$ for which the large balls $B(Q_i^k)$ have radius less than $r_E$, which allows for the application of \eqref{eq: CH-inequality}.  Since $E$ is a compact set and $f$ is continuous, $f(E)$ is also compact. Hence, by covering $f(E)$ by finitely many balls of radius $\de^{k_E d/D}/3$ and using the stability of the Minkowski dimension under finite unions (see p.~18, 20 in \cite{FraserBook}), we can further assume that $$\diam f(E) <2\de^{k_E d/D}/3 <\de^{k_E d/D}.$$ 
		
		Let $r \in (0,\diam f(E))$ and $k_r \geq k_E$ be the largest integer such that
		\begin{equation}\label{eq:def_k_r}
			\de^{\frac{(k_r+1)d}{D}}< r \leq \de^{\frac{k_rd}{D}},
		\end{equation} which exists because of the assumption $\diam f(E)<\de^{k_E d/D}$.
	
		For any $k\geq k_r$, we call a dyadic cube $Q_i^k$ of level $k$ that intersects $E$ a $k_r$-\textbf{major} cube if $\diam f(Q_i^k)\geq r$ and $k_r$-\textbf{minor} otherwise. Note that if $Q_i^k$ is $k_r$-major, by uniform continuity of $f$, there is some higher level $\ell \geq k$ for which there are no $k_r$-major cubes $Q_j^\ell$ lying in $Q_i^k$. We will count the number of $k_r$-minor cubes that occur if we increase the level by the smallest number necessary to only have $k_r$-minor sub-cubes of each $Q_i^k$ intersecting $E$. This collection of $k_r$-minor cubes will then result in a covering of $E$, which under $f$ yields a covering of $f(E)$ by sets of diameters at most $r$. An upper bound for the number of these sets also bounds $N(f(E), r)$ from above.
		
		Since $d>d_E$, by Proposition \ref{prop:boxdyadic} and \eqref{eq:def_k_r} there is  $C>0$ such that
		\begin{equation}\label{eq:dimB_minor-bound}
			N_{k_r}(E)\leq C \de^{-k_r d}\leq C r^{-D}.
		\end{equation}
		
		Note that if all such $Q_i^k$ are minor, then their images under $f$ are all sets of diameter at most $r$ covering $f(E)$ and \eqref{eq:dimB_minor-bound} provides the bound $N(f(E),r)\leq C r^{-D}$, concluding the proof.
		
		Denote by $M(k)$ the number of $k_r$-major cubes of level $k\geq k_r$ intersecting $E$ and suppose $M(k)>0$. The finite collection $\{B_i^k\}_{i\in I_k}$ also covers $E$ and has the property that $\frac{c_0}{6 C_0}B_i^k\cap \frac{c_0}{6 C_0}B_j^k=\emptyset$ for all distinct $i,j\in I_k$ by Theorem \ref{thm:Dydadic} (iii). Hence, by $f\in CH^{p,\alpha}$ and \eqref{eq: CH-inequality} we have
		$$
		\diam f(B_i^k)\leq C_i (\diam B_i^k)^\alpha,
		$$ 
        for all $i\in I_k$ for which $Q_i^k$ is $k_r$-major. But $f(Q_i^k)\subset f(B_i^k)$, so $\diam f(B_i^k)\geq r$ and due to \eqref{eq:def_k_r} we have for the above inequality that
		$$
		\de^{\frac{p d}{D}} \de^{\frac{k_r p d}{D}}\leq C_i^p (4C_0)^{p\alpha} \de^{kp\alpha}.
		$$ 
        Summing over all such $k_r$-major cubes of level $k$, i.e. over all corresponding $i\in I_k^+$, where $I_k^+ := \{ i\in I_k : Q_i^k \,\, \text{is} \,\, k_r\text{-major cube} \}$, we get
		$$
		M(k) \de^{\frac{k_r p d}{D}} \leq \de^{\frac{-p d}{D}}(4C_0)^{p\alpha}\sum_{i\in I_k^+}C_i^p \de^{kp\alpha}.
		$$ However, by $f\in CH^{p,\alpha}$ we have $\sum_{i\in I_k^+}C_i^p\leq C_E$. Summing over all levels $k\geq k_r$ and keeping in mind that the sequence $(M(k))_{k\geq k_r}$ is eventually $0$, we get
		$$
		\sum_{k=k_r}^{\infty} M(k) \de^{\frac{k_r p d}{D}} \leq \de^{\frac{-p d}{D}}(4C_0)^{p\alpha} C_E \sum_{k=k_r}^{\infty} (\de^{p\alpha})^k.
		$$ Since the series on the right-hand side is geometric, the above implies
		$$
		\sum_{k=k_r}^{\infty} M(k) \leq \tilde{C} \de^{\frac{-k_r p d}{D}} \de^{k_r p\alpha},
		$$ for some $\tilde{C}=\tilde{C}_{\delta,p, \alpha, d}>0$ that does not depend on $r$. But by definition of $D$, we have 
        $$
        \frac{-k_r p d}{D}+k_r p\alpha= -k_r d,
        $$ hence		
		\begin{equation}\label{eq:dimB_MajorNumber}
		\sum_{k=k_r}^{\infty} M(k) \leq \tilde{C} \de^{-k_r d}.
		\end{equation}
		Note that by (i) and (iii) of Theorem \ref{thm:Dydadic} and by doubling property of $X$, in each cube $Q_i^k$ there are at most $N_d$ cubes of level $k+1$, where $N_d$ only depends on the doubling constant of $X$ and the constants $c_0, C_0$. Hence, \eqref{eq:dimB_MajorNumber} implies that the number of minor cubes we get inside all major cubes is at most $N_d  \tilde{C} \de^{-k_r d}$. All these cubes, along with the $k_r$-minor cubes of level $k_r$, provide a covering of $E$ that when mapped under $f$ yield a covering of $f(E)$ by sets of diameter at most $r$, whose number by \eqref{eq:dimB_minor-bound} and \eqref{eq:dimB_MajorNumber}  cannot exceed
		$$
		C r^{-D}+ N_d  \tilde{C} \de^{-k_r d} \leq C' r^{-D}.
		$$ Thus, $N(f(E),r) \leq C' r^{-D}$. Since $r$ was arbitrary, the proof is complete.
    \end{proof}

\subsection{Newtonian-Sobolev mappings}\label{subsec:Newtonian}
Let $(X,d,\mu)$ be a proper
%
%
metric space with locally $Q$-homogeneous measure $\mu$, supporting a $Q$-Poincar\'e inequality with data $\la\geq 1$,
and let $(Y,d_Y)$ be an arbitrary metric space. In order to show that continuous mappings with locally integrable upper gradients are compactly H\"older, we need the following properties.

\begin{lemmaA}\label{Le:MS-inequality}
    Let $p > Q$ and  $f: X \rightarrow Y$ be a continuous mapping with upper gradient $g\in L^p_{\loc}(X)$. For any compact $K\subset X$ there are $C_K\geq 1$, $R_K>0$ such that for all balls $B=B(x,r)\subset K$ with $x\in K$, $r<R_K$ we have
    	\begin{equation}\label{eq:MS_forSob}
    	    |f(x)-f(y)|\leq C_K (\diam B)^{Q/p}|x-y|^{1-Q/p} \left( \Barint_{4 \la B} g^p d\mu \right)^{1/p},
    	\end{equation} for all $x, y\in B$.
\end{lemmaA}
The above property is essentially a Morrey-Sobolev inequality counterpart for Newtonian-Sobolev maps, see Theorem 9.2.14 in \cite{Juha-Jeremy-etc-book}. Note that the assumption in Theorem 9.2.14 that $X$ is quasiconvex is not necessary for our context, since we apply the result on a large compact ball containing our set of interest (see \eqref{eq:E in comp_ball} below, and Theorem 8.3.2 in \cite{Juha-Jeremy-etc-book}).

\begin{lemmaA}\label{Le:Make_disjoint_integrals}
	Let $1\leq q < p$ and $\tau\in (0,1)$. For each compact set $K\subset X$ there is a constant $C_K'\geq 1$ and a radius $R_K'>0$  such that for all $g\in L^q(K)$, there is a Borel function $\tilde{g} \in L^{p/q}(K) \subset L^1(K)$ so that
	\begin{equation}\label{eq:MaximalFunIneq}
	    \Barint_{B(x,r)}g^q d\mu \leq C_K' \Barint_{B(x,\tau r)}\tilde{g}  d\mu,
	\end{equation} for all $x\in K$ and $0<r<R_K'$.
\end{lemmaA}
The above is a corollary of the Maximal Function Theorem, see Chapter 2 in \cite{hei:lectures} and Lemma 3.3 in \cite{BaloghAGMS} for a proof. The reason we are restricted to $q$ strictly less than $p$ is due to $\tilde{g}$ being in fact a conveniently restricted maximal function of $g^q$.

Suppose $f:X\rightarrow Y$ is continuous with an upper gradient $g\in L^p_{\loc}(X)$. Note that for $1<q<p$ we also have $g\in L^q_{\loc}(X)$. We plan on using Lemmas \ref{Le:MS-inequality} and \ref{Le:Make_disjoint_integrals} to show that $f\in CH^{q,1-Q/q}$ for all $q\in (Q,p)$, which is enough to achieve \eqref{eq:SobBoxdistortion}.

    \begin{proof}[Proof of Theorem \ref{thm:NewtonDim}]
		Let $E\subset X$ be compact and $\eps\in (0,1)$. Note that the centers of balls that cover $E$ in the definition of compactly H\"older mappings do not lie in $E$ necessarily, while it is a requirement for the inequalities in Lemmas \ref{Le:MS-inequality} and \ref{Le:Make_disjoint_integrals}. Thus, we need to apply these properties to a potentially larger compact set than $E$. Using a similar argument to that in the proof of Proposition \ref{prop:boxdyadic}, it can be shown that
        \begin{equation}\label{eq:E in comp_ball}
            E\cup\left(\bigcup_{x\in E}B(x,1/4)\right) \subset B(x_E, \diam E+1/2),
        \end{equation} for some $x_E\in E$. Thus, if $K$ is the closure of $B(x_E,  \diam E+1/2)$, and $\mathcal{B}$ is a covering of $E$ by balls of radius at most $1/10$, then all elements of $\mathcal{B}$ lie entirely in $K$, along with their centers.
        We plan on applying the above Lemmas for $\tau:=\eps/4\la$ on $K$, which is a compact set due to $X$ being proper, and set
        \begin{equation}\label{eq:SobIsCH_radius}
            r_E:= \min\left\{ \frac{R_K}{4\la}, \frac{R_K'}{4\la}, \frac{1}{10}, \frac{R_\text{hom}(K)}{10} \right\} ,
        \end{equation} where $R_K, R_K'$ are the radii in Lemmas \ref{Le:MS-inequality} and \ref{Le:Make_disjoint_integrals}, respectively,  $\la$ is the data from the $Q$-PI of $X$ and $R_{\text{hom}}(K)$ is the local homogeneity radius of $X$ for the compact set $K$.
        
		Suppose $\{B(x_i,r)\}_{i\in I}$ is a  cover of $E$ with $r<r_E$ and $B(x_i,\eps r)\cap B(x_j,\eps r)=\emptyset$ for all distinct $i,j\in I$. We will show that an inequality of the form \eqref{eq: CH-def-inequality} holds. By \eqref{eq:SobIsCH_radius} and Lemma \ref{Le:MS-inequality}, we have for all balls $B_i:=B(x_i,r)$ the inequality
		$$
		|f(x)-f(y)|\leq C_K (\diam B_i)^{Q/q}|x-y|^{1-Q/q} \left( \Barint_{4 \la B_i} g^q d\mu \right)^{1/q},
		$$ for all distinct $x,y\in B_i$, where $C_K$ may also depend on $q$. We set $\alpha:=1-Q/q$ and divide with $|x-y|^{1-Q/q}=|x-y|^\alpha$ to get
        $$
        |f|_{\alpha,B_i}\leq C_K (\diam B_i)^{Q/q} \left( \Barint_{4 \la B_i} g^q d\mu \right)^{1/q}.
        $$
  
        Note that by \eqref{eq:lower-AR} we could bound uniformly from above the term $C_K (\diam B_i)^{Q/q}/\mu(4\la B_i)^{1/q}$ on the right, and leave just the integral terms to depend on $i$. However, setting $C_i$ to be the uniform constant times the integral on the right hand side of the above inequality would not be enough to prove the compactly H\"older property. The reason for this is the potentially large overlap the balls $4 \la B_i$ might have, contradicting any upper bound $C_E$ on the $q$-sum of the constants $C_i$. To avoid this issue, we apply Lemma \ref{Le:Make_disjoint_integrals} for $\tau=\eps/4\la$ on all integrals on the right hand side, which is possible due to the choice \eqref{eq:SobIsCH_radius} and  $r<r_E$, and we get
        $$
        |f|_{\alpha,B_i}\leq C_K (C_K')^{1/q} (\diam B_i)^{Q/q} \left( \Barint_{\eps B_i} \tilde{g} d\mu \right)^{1/q}.
		$$ 
        Since $X$ is locally $Q$-homogeneous, by \eqref{eq:lower-AR} the quotient $\frac{(\diam B_i)^Q}{\mu(\eps B_i)}$ is at most $\tilde{C}:=C_{\text{hom}}(K) \left(\frac{1}{\eps}\right)^Q$. Hence, 
        due to $B(x_i,\eps r)\cap B(x_j,\eps r)=\emptyset$ and 
        $$
        \sum_{i\in I} \int_{\eps B_i} \tilde{g} d\mu =  \int_{\bigcup\limits_{i\in I} \eps B_i} \tilde{g} d\mu \leq  \int_{K} \tilde{g} d\mu,
        $$ 
        there is $C_E=C_K^q C_K' \tilde{C} \left( \int_{K} \tilde{g} d\mu \right)<\infty$ such that
        $$
        \sum\limits_{i\in I} |f|_{\alpha,B_i}^q \leq \sum\limits_{i\in I} C_K^q C_K' \tilde{C}  \int_{\eps B_i} \tilde{g} d\mu  \leq C_E.
        $$
        Since $E$ and $\eps$ were arbitrary, and $K$ depends only on $E$, this implies that $f\in CH^{q,1-Q/q}$ for all $q\in (Q,p)$ as needed. By Theorem \ref{thm:main_Holder} we have that
        $$
        \dim_B f(E)\leq \frac{q d_E}{q-Q+d_E},
        $$ for all $q\in (Q,p)$, which implies \eqref{eq:SobBoxdistortion} for $q\rightarrow p$.
  
    \end{proof}
    
\subsection{Quasisymmetric mappings}
    Suppose $Q>1$ and $(X,d,\mu)$ is a proper, $Q$-Ahlfors regular metric measure space that supports a $p_0$-PI for $p_0\in(1,Q)$, and $(Y,d_Y)$ is $Q$-Ahlfors regular. Let $f:X\rightarrow Y$ be a quasisymmetric mapping and $E\subset X$ bounded.

    \begin{proof}[Proof of Corollary \ref{cor:QS_Dim}]
        Under the above assumptions on $X$ and $Y$, Heinonen and Koskela showed in \cite{HeinKoskQCbegin} a higher integrability result similar to that by Gehring in $\R^n$ \cite{Gehring73}. More specifically, $f$ has an upper gradient lying in $L^p_{\loc}(X)$, for some $p>Q$ (see Theorem 9.3 in \cite{HeinKoskQCbegin} 
        %
        ). The membership of $f$ in $CH^{q,1-Q/q}$ for all $q\in (Q,p)$ and the upper bound on $\dim_B f(E)$ then follow from Theorem \ref{thm:NewtonDim}. By an identical argument on the inverse of $f$, which is also a quasisymmetric mapping, the lower bound in \eqref{eq:QSboxdistortion} is also determined (see, for instance, Proposition 10.6 in \cite{hei:lectures}).
    \end{proof}

\section{Final Remarks}\label{sec:Final Remarks}
%
%
%
%
%
%

There are ways to reduce the assumptions on some of our main results. For instance, due to $X$ being proper and $E$ bounded, we could only require that $X$ supports a local Poincar\'e inequality in Theorem \ref{thm:NewtonDim} and the outcome would still be true. Moreover, instead of $f$ having an upper gradient $g\in L^p_{\text{loc}}$ in Theorem \ref{thm:NewtonDim}, the result would hold even if $g$ was a $p$-weak upper gradient instead (see Chapter 6 in \cite{Juha-Jeremy-etc-book}). Additionally, the requirements for $f$ and $X$, $Y$ in Corollary \ref{cor:QS_Dim} can also be reduced. It would be enough to have that $f$ is a local quasisymmetric mapping, $X$ and $Y$ are locally $Q$-homogeneous (instead of Ahlfors regular), as long as $X$ is a Loewner space and $Y$ is linearly locally connected (see \cite{hei:lectures}, \cite{HeinKoskQCbegin}). However, for these generalizations we would have to define more notions and state strong analytical results only to slightly generalize our assumptions. Hence, in an effort to keep the manuscript more accessible to a broader mathematical audience, we invite the interested reader to fill in the details and complete this generalization.

It should be noted that in the proofs of our main result for compactly H\"older mappings, the definition was not fully utilized. More specifically, for the main argument we employed \eqref{eq: CH-inequality}, and not \eqref{eq: CH-def-inequality} directly. Hence, one can define the \textbf{weak compactly H\"older} mappings as in Definition \ref{def: CH maps}, but by replacing \eqref{eq: CH-def-inequality} with \eqref{eq: CH-inequality} and the assumption on the $p$-sum of the sequence $C_i$. If we denote this class by $CH^{p,\alpha}_w$, we have the inclusion $CH^{p,\alpha}\subset CH_w^{p,\alpha}$, and Theorem \ref{thm:main_Holder} would still be true for $f\in CH^{p,\alpha}_w$ as well.

While Corollary \ref{cor:QS_Dim} is quantitative in the sense that the bounds only eventually depend on $\eta$, the dependence of the Newton-Sobolev exponent $p$ on $\eta$ is very implicit. This $p$ is called an ``exponent of higher integrability" of $f$ and has attracted a lot of interest due to the implications it carries for the theory of quasiconformal and quasisymmetric mappings. In fact, other than the case of $X=Y=\R^2$ with the usual metric and measure, which is due to Astala \cite{Astala}, we do not have an explicit formula not even for quasiconformal mappings on $\R^n$ for $n>2$. See the Introduction and Remark 4.4 in \cite{OurQCspec} for a brief discussion of higher integrability exponents in the context of dimension distortion and \cite{im:gft} for a complete exposition.

Recall that we gave the metric definition of quasisymmetric mappings in Section \ref{sec:background}, while there are two more definitions (analytic, geometric). There are, in general, requirements to be imposed on the metric spaces $X$ and $Y$ in order to have all definitions equivalent. However, the strictly metric nature of the definition of compactly H\"older mappings motivates the following question.

\begin{question}\label{que: metric qs implies CH}
    Suppose $X$ is a doubling, connected metric space and $Y$ is an arbitrary metric space. If $f:X\rightarrow Y$ is an $\eta$-quasisymmetric embedding, under what metric assumptions on $X$, $Y$ and/or on $\eta$ are there constants $p>1$, $\alpha\in (0,1)$, appropriately dependent on $\eta$, such that $f$ is a (weak) $(p,\alpha)$-compactly H\"older mapping?
\end{question}
Note that $X$ needs to be connected to avoid examples of quasisymmetric mappings that are not even locally H\"older continuous (see Corollary 11.5 and the following discussion in \cite{hei:lectures}). In addition, despite the example of the $\alpha$-snowflake quasisymmetric mapping $\id_\alpha$ mentioned in the Introduction not being $(p,\alpha)$-compactly H\"older for any $p$, there might still exist $p=p(\alpha)$ and $\beta= \beta(\alpha)\neq \alpha$ for which $\id_\alpha$ is $(p,\beta)$-compactly H\"older. Such a relation could imply Minkowski dimension distortion bounds for quasisymmetric mappings directly from Theorem \ref{thm:main_Holder}, without any measure-theoretic regularity conditions on $X$ and $Y$. 



\bibliographystyle{acm}
\bibliography{SobolevDim}
\end{document}